\documentclass[11pt,a4paper,leqno]{amsart}
\usepackage[a4paper, left=2cm, right=2cm, top=1.5cm, bottom=1.5cm]{geometry}

\usepackage{amsmath,amsfonts,amssymb,latexsym,amstext,amsthm}
\usepackage{hyperref}
\usepackage[utf8]{inputenc}
\usepackage{microtype}
\usepackage{graphicx}
\newcommand{\NN}{\mathbb{N}}
\newcommand{\CC}{\mathbb{C}}
\newcommand{\RR}{\mathbb{R}}
\newcommand{\EE}{\mathbb{E}}
\newcommand{\Oo}{\mathcal{O}}

\newcommand{\ord}{\textrm{ord}_t}

\newcommand{\bd}[1]{\boldsymbol{#1}}

\renewcommand{\epsilon}{\varepsilon}
\renewcommand{\phi}{\varphi}

\allowdisplaybreaks

\newtheorem{thm}{Theorem}
\newtheorem{lemma}{Lemma}
\newtheorem{prop}{Proposition}
\theoremstyle{definition}
\newtheorem{remark}{Remark}
\newtheorem{defi}{Definition}
\newtheorem*{defi*}{Definition}
\newtheorem{example}{Example}
\usepackage{enumerate}

\begin{document}

\title{Gevrey estimates of formal solutions for certain moment partial differential equations with variable coefficients}

\author{Maria Suwi\'nska}

\address{
	        	Faculty of Mathematics and Natural Sciences,
				College of Science\\
				Cardinal Stefan Wyszy\'nski University\\
				W\'oycickiego 1/3,
				01-938 Warszawa, Poland, ORCID: 0000-0002-4493-6519
}
\email{m.suwinska@op.pl}         

\begin{abstract}
The goal of this paper is to investigate Gevrey properties of formal solutions of certain generalized linear partial differential equations with variable coefficients. In particular, we extend the notion of moment partial differential equations to include differential operators that are not connected with kernel functions. Using the modified version of Nagumo norms and the properties of the Newton polygon we compute the Gevrey estimate for formal solutions of such generalized partial differential equations.  
\end{abstract}

\maketitle
\sloppy

\section{Introduction}

Over the last decade many studies have been devoted to the topic of
the Gevrey properties of formal series solutions of linear partial
differential equations. Main topics of interest include inhomogeneous
PDEs with constant coefficients \cite{BY}, the heat equation and
its various generalizations \cite{BLR,Remy1,Remy2} and PDEs with
time-dependent coefficients \cite{TY}. Those results have also been
generalized to the case of moment partial differential equations with
constant and time-dependent coefficients \cite{Mic1,Mic2,MicS}.

In this paper we aim to further generalize the notion of moment partial
differential equations. To this end, we define Gevrey-type sequences
and operators, which maintain many of the characteristics of the moment
functions and moment differential operators, respectively. More details on the subject are given in Section
\ref{Sec2}. 

We focus our attention on a general Cauchy problem of the form:
\begin{equation}
\left\{ \begin{aligned}P\left(\partial_{m_{0},t},\partial_{m_{1},z_{1}},\dots,\partial_{m_{N},z_{N}}\right)u(t,\bd z) & =f(t,\bd z)\\
\partial_{m_{0},t}^{j}u(0,\bd z) & =\phi_{j}(\bd z)\textrm{ for }0\leq j<M
\end{aligned}
\right.,\label{eq:main_0}
\end{equation}
where $P\left(\partial_{m_{0},t},\partial_{m_{1},z_{1}},\dots,\partial_{m_{N},z_{N}}\right)$
is a linear operator with coefficients depending on both variables
$t\in\CC$ and $\bd z\in\CC^{N}$, and $\partial_{m_{0},t},\partial_{m_{1},z_{1}},\dots,\partial_{m_{N},z_{N}}$
are Gevrey-type differential operators (see Definition \ref{def:gevrey_op}).
Our goal is to connect the Gevrey order of its solution with the orders
of the inhomogeneity $f$ and the variable coefficients.

While analyzing an equation very similar to the ones studied in \cite{TY,MicS},
we mostly draw inspiration from \cite{BLR,Remy1}, where the concept of Nagumo norms is utilized. Below we recall the definition used in papers mentioned above:

\begin{defi*}
Let us consider a function $f$ holomorphic on a disc $D_{R}=\left\{ z\in\CC:\ |z|<R\right\} ,$
$p\geq0$ and $0<r<R$. Then \emph{the Nagumo norm $\|f\|_{p,r}$
of $f$} is defined by
\begin{equation*}
\|f\|_{p,r}=\sup_{z\in D_{r}}|f(z)(r-|z|)^{p}|.
\end{equation*}
\end{defi*}
The usefulness of this family of norms lays in their
properties. Namely, for any two functions $f_{1},f_{2}$ holomorphic
on $D_{R}$, non-negative numbers $p,q$ and $0<r<R$ we have:
\begin{itemize}
\item $\|f_{1}f_{2}\|_{p+q,r}\leq\|f_{1}\|_{p,r}\|f_{2}\|_{q,r}$, 
\item $\|\partial_{z}f_{1}\|_{p+1,r}\leq C(p+1)\|f_{1}\|_{p,r}$ for a certain
constant $C>0$,
\item $|f_{1}(z)|\leq\|f_{1}\|_{p,r}(r-|z|)^{-p}$ for all $z\in D_{r}$.
\end{itemize}
Especially the second attribute is very important, because it enables
us to estimate $\partial_z f$ by $f$ without making shrinking of
the domain necessary. Unfortunately, the standard definition mentioned above
could not be used to determine the Gevrey order of the solution of
(\ref{eq:main_0}). This fact made creating a new tool with similar
properties necessary and so certain modifications were utilized to
generalize the Nagumo norms in a way more fitted to our needs. Indeed,
it has been proven in Lemmas \ref{lem:nagumo_prop}, \ref{lem:nagumo_der}
and \ref{lem:sup_norm} that the properties listed above hold for
the modified Nagumo norms defined in Section \ref{Sec2}. Moreover, in some cases
both norms coincide for $z\in\CC$.

The main result of the paper, Theorem \ref{thm:gevrey}, shows that
under certain additional assumptions the Gevrey order of the formal
solution of (\ref{eq:main_0}) is equal to the reverse of the slope
of the first non-horizontal segment in the Newton polygon for $P$.
This outcome is a direct generalization of results found in \cite{BLR,MicS,Remy1,TY}.

The paper is structured as follows. In Section \ref{Sec2} we first
establish the notation and then move on to defining Gevrey-type sequences
and operators and presenting their characteristics. Some of the most
important examples are listed there as well. We also recall the definition
of the Newton polygon and alter it to better accommodate the structure
of equation (\ref{eq:main_0}). Section \ref{Sec3} is entirely devoted
to the modified Nagumo norms. Several important properties of these
norms are presented. We also go into further detail on their connection
with the Gevrey order of a function. Section
\ref{Sec4} contains all conditions, under which we would like to consider our initial equation. After that we present
the proof of Proposition \ref{prop:norm_est}, from which the main
result of this paper follows. Ideas for further study are included
in Section \ref{Sec5}.

\section{Preliminaries}\label{Sec2}

\subsection{Notation}

Let us introduce the notation that will be used throughout this paper.
First, note that by $\NN_{0}$ we shall denote the set of all non-negative
integers and $\NN=\NN_{0}\setminus\left\{ 0\right\} .$

For any positive integer $N$, by $D_{R}^{N}$ we denote a polydisc
with a center at $0$ and a radius $R>0$, i.e., a set 
\begin{equation*}
D_{R}^{N}=\left\{ \bd z=(z_{1},\dots,z_{N})\in\CC^{N}:\ |z_{j}|<R\textrm{ for }j=1,\dots,N\right\} .
\end{equation*}

If $\Omega\subset\CC^{N}$ and a function $f$ is holomorphic on the
set $\Omega$ then we write $f\in\Oo(\Omega).$

For any function $f\in\Oo(D_{R})$ by $\ord(f)$ we denote the order
of zero of the function $f$ at $t=0$.

A space of all formal power series $\hat{u}(t)=\sum_{n=0}^{\infty}u_{n}t^{n}$
with coefficients from a given Banach space $\mathbb{E}$ will be
denoted by $\mathbb{E}[[t]]$. Throughout this paper we will restrict
ourselves to the case when $\mathbb{E}$ is the space $\Oo(D_{R}^{N})\cap C(\overline{D}_{R}^{N})$
of all functions holomorphic on a polydisc $D_{R}^{N}\subset\CC^{N}$ and continuous on its closure, equipped with the standard norm $\|f\|=\sup_{\bd\zeta\in D_{R}^{N}}|f(\bd\zeta)|.$

For all multi-indices $\bd\alpha=(\alpha_{1},\alpha_{2},\dots,\alpha_{N})$,
$\bd\beta=(\beta_{1},\beta_{2},\dots,\beta_{N})$ in $\NN_{0}^{N}$,
points $\bd z=(z_{1},z_{2},\dots,z_{N})\in\CC^{N}$ and $\bd s=(s_{1},s_{2},\dots,s_{N})\in\RR^{N},$
and $A\in\RR$ the following notation will be used: 
\begin{align*}
\bd\alpha+\bd\beta=(\alpha_{1}+\beta_{1},\alpha_{2}+\beta_{2},\dots,\alpha_{N}+\beta_{N})\qquad & \bd\alpha\cdot\bd\beta=\sum_{j=1}^{N}\alpha_{j}\beta_{j}\\
\bd\alpha\leq\bd\beta\iff\alpha_{j}\leq\beta_{j}\textrm{ for }j=1,\dots,N\qquad & |\bd\alpha|=\sum_{j=1}^{N}\alpha_{j}\\
\bd z^{\bd s}=z_{1}^{s_{1}}z_{2}^{s_{2}}\dots z_{N}^{s_{N}}\qquad & A\bd\alpha=(A\alpha_{1},A\alpha_{2},\dots,A\alpha_{N})\\
\bd0=(0,0,\dots,0).\qquad & 
\end{align*}

Let $a(\bd z)=\sum_{\bd\gamma\in\NN_{0}^{N}}a_{\bd\gamma}\bd z^{\bd\gamma}$
and $b(\bd z)=\sum_{\bd\gamma\in\NN_{0}^{N}}b_{\bd\gamma}\bd z^{\bd\gamma}$
be two formal power series. Then we call $b(\bd z)$ \emph{a majorant}
of $a(\bd z)$ and write $a(\bd z)\ll b(\bd z)$ if $|a_{\bd\gamma}|\leq b_{\bd\gamma}$
for every $\bd\gamma\in\NN_{0}^{N}.$

\subsection{Gevrey-type sequences and operators}

\begin{defi}\label{def:gevrey_seq} Let $m(n)$ denote a sequence
of positive numbers such that $m(0)=1$. Then $m(n)$ will be called
\emph{a Gevrey-type sequence of order $s\geq0$ }if there exist
constants $a,A>0$ satisfying the condition 
\begin{align}
a^{n}n!^{s} & \leq m(n)\leq A^{n}n!^{s}\textrm{ for any }n\in\NN_{0}.\label{eq:gevrey_seq}
\end{align}
\end{defi}

Note that Definition \ref{def:gevrey_seq} encompasses a wide variety
of sequences. In particular, (\ref{eq:gevrey_seq}) describes one
of the main characteristics of so-called moment functions, which
are in turn directly connected to kernel functions. For more information
on the topic we refer the Reader to \cite[Section 5.5]{B}. In this
paper we depart from this concept, but many elements used in the theory
of moment functions and moment differential operators can be easily
adapted to this more general family of sequences. We devote the next
part of this chapter to presenting their various properties.

\begin{prop}[see {\cite[Theorems 31 and 32]{B}}] Let $m_{1}$
and $m_{2}$ be two Gevrey-type sequences of non-negative orders $s_{1}$
and $s_{2}$, respectively. Then $m_{1}\cdot m_{2}$ is a Gevrey-type
sequence of order $s_{1}+s_{2}$. Moreover, if $s_{1}>s_{2}$ then
$\frac{m_{1}}{m_{2}}$ is a Gevrey-type sequence of order $s_{1}-s_{2}.$
\end{prop}

Below we present several examples of Gevrey sequences.

\begin{example}\label{ex:gevrey_seq}
\leavevmode
\makeatletter
\@nobreaktrue
\makeatother
\begin{itemize}
\item Function $\Gamma_{s}(n)=\Gamma(1+sn)$, where $s>0$, is a Gevrey-type
sequence of order $s$. 
\item Let us put $[n]_{q}=1+q+\ldots+q^{n-1}=\frac{q^{n}-1}{q-1}$ for a
fixed $q\in(0,\,1)$. Then a sequence given by the formula 
\begin{equation*}
[n]_{q}!=\begin{cases}
[n]_{q}\cdot[n-1]_{q}\cdot\ldots\cdot[1]_{q} & \textrm{ for }n\geq1\\
1 & \textrm{ for }n=0
\end{cases}
\end{equation*}
is a Gevrey-type sequence of order $0$. Indeed, it is easy to see
that 
\begin{equation*}
1\leq[n]_{q}!\leq\frac{1}{(1-q)^{n}}\textrm{ for any }n\in\NN_{0}.
\end{equation*}
\end{itemize}
\end{example}

We will also use a special class of sequences defined below.

\begin{defi} If $m(n)$ is a sequence of positive numbers such that
$m(0)=1$ and $s\geq0$ then we shall call $m(n)$ \emph{a regular
Gevrey-type sequence of order $s$ }if there exist constants $a,A>0$
satisfying the condition 
\begin{align}
a(n+1)^{s} & \leq\frac{m(n+1)}{m(n)}\leq A(n+1)^{s}\textrm{ for any }n\in\NN_{0}.\label{eq:gevrey_reg}
\end{align}
\end{defi}

\begin{remark} It is easy to observe that every regular Gevrey-type
sequence of order $s$ is a Gevrey-type sequence of the same order
and for the same constants $a$ and $A$.
\end{remark}

Regular Gevrey-type sequences satisfy the following properties:
\begin{lemma}[compare {\cite[Lemma 2.1]{MicS}}]
\leavevmode
\makeatletter
\@nobreaktrue
\makeatother
\begin{enumerate}[(i)]
\item The class of regular Gevrey-type sequences is closed under multiplication
and division.\label{enu:gevrey_reg_class1} 
\item The class of regular Gevrey-type sequences contains the sequence $\Gamma_{s}(n)=\Gamma(1+sn)$.\label{enu:gevrey_reg_class2} 
\end{enumerate}
\end{lemma}

\begin{defi}[see \cite{BY}]\label{def:gevrey_op} Let $m$ be a Gevrey-type sequence.
Then we define a \emph{Gevrey-type differential operator} $\partial_{m,t}\colon\mathbb{E}[[t]]\to\mathbb{E}[[t]]$
by the formula: 
\begin{equation*}
\partial_{m,t}\left(\sum_{n=0}^{\infty}\frac{u_{n}}{m(n)}t^{n}\right):=\sum_{n=0}^{\infty}\frac{u_{n+1}}{m(n)}t^{n}.
\end{equation*}
\end{defi}

A couple of examples of Gevrey-type differential operators can be found below:

\begin{example}
\leavevmode
\makeatletter
\@nobreaktrue
\makeatother
\begin{itemize}
\item For $m(n)=n!$ operator $\partial_{m,t}$ coincides with the standard
derivative $\partial_{t}$.
\item Consider a Gevrey-type sequence $\Gamma_{s}(n)$ defined in Example
\ref{ex:gevrey_seq}. Then we have 
\begin{equation*}
\partial_{t}^{s}u(t^{s})=\left(\partial_{\Gamma_{s},w}u\right)(t^{s})
\end{equation*}
with $\partial_{t}^{s}$ being the Caputo fractional derivative.
\item Suppose that $m(n)=[n]_{q}!$ for a certain $q\in(0,\,1).$ We have
established in Example \ref{ex:gevrey_seq} that $m(n)$ is a Gevrey-type
sequence of order $0$. We will show that the $q$-difference differential
operator given by
\begin{equation*}
D_{q,t}u(t)=\frac{u(qt)-u(t)}{qt-t}
\end{equation*}
is equal to the Gevrey-type operator $\partial_{m,t}$ for any function
$u(t)=\sum_{n=0}^{\infty}\frac{u_{n}}{[n]_{q}!}t^{n}$. Indeed, let
us note that
\begin{equation*}
D_{q,t}t^{n}=\frac{q^{n}t^{n}-t^{n}}{qt-t}=[n]_{q}t^{n-1}\textrm{ for any }n\in\NN.
\end{equation*}
 Hence, 
\begin{equation*}
D_{q,t}u(t)=\sum_{n=1}^{\infty}\frac{u_{n}}{[n]_{q}!}[n]_{q}t^{n-1}=\sum_{n=1}^{\infty}\frac{u_{n}}{[n-1]_{q}!}t^{n-1}=\partial_{m,t}\left(\sum_{n=0}^{\infty}\frac{u_{n}}{[n]_{q}!}t^{n}\right).
\end{equation*}
\end{itemize}
\end{example}

We can now recall the basic definition of the Gevrey order.

\begin{defi} \label{def:gevrey_order} The formal series $\hat{u}(t)=\sum_{n=0}^{\infty}u_{n}t^{n}\in\mathbb{E}[[t]]$
is of Gevrey order $s$ if and only if there exist $B,C>0$ such that
\begin{gather*}
\|u_{n}\|_{\EE}\leq BC^{n}n!^{s}\quad\textrm{for}\quad n\in\NN_{0}.
\end{gather*}
Then we write that $\hat{u}(t)\in\mathbb{E}[[t]]_s$.
\end{defi}

\begin{remark}

Definition \ref{def:gevrey_order} implies that any formal power series
of Gevrey order $0$ is convergent.

\end{remark}

\subsection{The Newton polygon}

The Newton polygon for linear partial differential operators with variable coefficients has been introduced in \cite{Y}. The notion was later extended in \cite{Mic2} to moment partial differential equations with constant coefficients with two variables. The definition presented below is based on \cite{MicS}, where the Newton polygon was defined for linear moment partial differential operators with time-dependent coefficients.

Let $m_{0},m_{1},\dots,m_{N}$ be Gevrey sequences of positive orders
$s_{0},s_{1},$ $\dots,$ $s_{N}$, respectively, and suppose that $\partial_{m_{0},t}$,
$\partial_{m_{1},z_{1}},\dots,\partial_{m_{N},z_{N}}$ are Gevrey-type
differential operators. For $\bd m=(m_{1},m_{2},\dots,m_{N})$ and
any multi-index $\bd\alpha\in\NN_{0}^{N}$ we use a notation $\partial_{\bd m,\bd z}^{\bd\alpha}:=\partial_{m_{1},z_{1}}^{\alpha_{1}}\partial_{m_{2},z_{2}}^{\alpha_{2}}\dots\partial_{m_{N},z_{N}}^{\alpha_{N}}.$
Suppose that $J\subset\NN_{0}$ and $A\subset\NN_{0}^{N}$ are finite
sets of indices. We shall consider an operator of the form: 
\begin{align}
P\left(t,\bd z,\partial_{m_{0},t},\partial_{\bd m,\bd z}\right) & =\sum_{(j,\bd\alpha)\in J\times A}a_{j,\bd\alpha}(t,\bd z)\partial_{m_{0},t}^{j}\partial_{\bd m,\bd z}^{\bd\alpha},\label{eq:gen_operator}
\end{align}
where $a_{j,\bd\alpha}(t,\bd z)$ are holomorphic with respect to
$t$ in an open neighborhood of zero for every $(j,\bd\alpha)\in J\times A.$
Furthermore let $\bd s=(s_{1},\dots,s_{N}).$

\begin{figure}[hbt]
\includegraphics[width=0.75\textwidth]{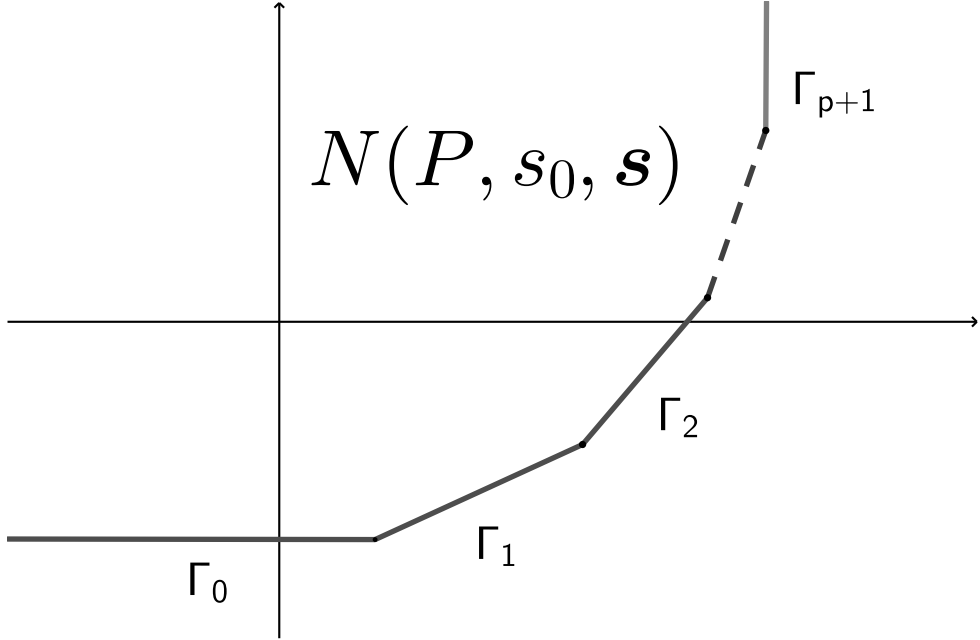} \caption{An example of a Newton polygon.\label{gfx:newton_polygon}}
\end{figure}

\begin{defi}

We define \emph{the Newton polygon for the operator $P$} given
by (\ref{eq:gen_operator}) as 
\begin{equation*}
N(P,s_{0},\bd s)=\textrm{conv}\left\{ \Delta\left(s_{0}j+\bd s\cdot\bd\alpha,\,\ord(a_{j,\bd\alpha})-j\right):\ (j,\bd\alpha)\in J\times A,\ a_{j,\bd\alpha}\not\equiv0\right\} ,
\end{equation*}
where $\Delta(a,b)=\left\{ (x,y)\in\RR^{2}:\ x\leq a,\,y\geq b\right\} $
for any $a,b\in\RR^{2}$.

\end{defi}

An example of a Newton polygon can be seen on Figure \ref{gfx:newton_polygon}.

\section{Modified Nagumo norms}\label{Sec3} 

\begin{defi} \label{def:theta_1dim}For any $a\in\NN_{0}$ and $s\geq0$
let us define a power series $\Theta_{s}^{(a)}(x)$ by the formula
\begin{equation}
\Theta_{s}^{(a)}(x):=\sum_{n=0}^{\infty}\frac{(n+a)!^{s}}{n!^{s}}x^{n}.
\end{equation}
\end{defi}

Using series $\Theta_{s}^{(a)}(x)$ we can define a modified version
of the Nagumo norm.

\begin{defi}\label{def:nagumo_norm} For any $\bd s\in[1,\,\infty)^{N}$,
$\bd\alpha\in\NN^{N}$, $0<r<R$ and $\bd z\in D_{r}^{N}$ we define
\emph{the modified Nagumo norm} $\|f(\bd z)\|_{\bd\alpha,r,\bd s}$
of $f\in\Oo(D_{R}^{N})$ by the formula: 
\begin{equation}
\|f(\bd z)\|_{\bd\alpha,r,\bd s}:=\inf\left\{ A\geq0:\ f(\bd z)\ll A\prod_{i=1}^{N}\frac{1}{r^{\alpha_{i}}(\alpha_{i}-1)!^{s_{i}}}\Theta_{s_{i}}^{(\alpha_{i}-1)}\left(\frac{z_{i}}{r}\right)\right\} .\label{eq:nagumo}
\end{equation}

\end{defi}

It is worth mentioning that for $N=1$ and $s=1$ Definition \ref{def:nagumo_norm}
coincides with the standard way of defining the Nagumo norm. Indeed,
let us take $q\in\NN$ and $0<r<R$. Then, for any function $f$,
(\ref{eq:nagumo}) takes form
\begin{equation*}
\|f(z)\|_{q,r,1}:=\inf\left\{ A\geq0:\ f(z)\ll A\frac{1}{r^{q}(q-1)!}\Theta_{1}^{(q-1)}\left(\frac{z}{r}\right)\right\} .
\end{equation*}
Note that
\begin{equation*}
\Theta_{1}^{(q-1)}\left(\frac{z}{r}\right)=\sum_{n=0}^{\infty}\frac{(n+q-1)!}{n!}\frac{z^{n}}{r^{n}}=\frac{\textrm{d}^{q-1}}{\textrm{d}u^{q-1}}\left.\left(\sum_{n=0}^{\infty}u^{n}\right)\right|_{u=r^{-1}z},
\end{equation*}
which gives us
\begin{equation*}
\sup_{|z|<\rho}|f(z)|\leq\frac{A}{r^{q}(q-1)!}\frac{r^{q}(q-1)!}{(r-|z|)^{q}}=\frac{A}{(r-|z|)^{q}}
\end{equation*}
for any $\rho<r$. Hence, $A\geq(r-\rho)^{q}\sup_{|z|<\rho}|f(z)|$.

\begin{remark} Note that $\|\cdot\|_{\bd\alpha,r,\bd s}$ is a norm
on $\Oo(D_{R}^{N})$ for any $\bd\alpha\in\NN^{N}$, $0<r<R$ and
$\bd s\in[1,\infty)^{N}$.
\end{remark}

\begin{remark} \label{rem:norm_0}
For the multi-index $\bd0$ we put $\|f(\bd z)\|_{\bd0,r,\bd s}$
equal to the standard norm in the space $\ell_{1}$ for any $0<r<R$
and $\bd s\in[1,\infty)^{N}$, i.e., if $f(\bd z)=\sum_{\bd\gamma\in\NN^N_0}f_{\bd\gamma}\bd z^{\bd\gamma}$ then
\begin{equation*}
\|f(\bd z)\|_{\bd 0,r,\bd s}=\sum_{\bd\gamma\in\NN^N_0}\left|f_{\bd\gamma}\right|r^{|\bd\gamma|}.
\end{equation*}
\end{remark}

\begin{lemma}\label{lem:nagumo_prop} Let $f,g\in\Oo\left(D_{R}^{N}\right)$,
$\bd\alpha,\,\bd\beta\in\NN^{N}$. Then for the modified Nagumo norm
defined by \eqref{eq:nagumo} the following properties hold: 
\begin{enumerate}[(i)]
\item $\|f(\bd z)\cdot g(\bd z)\|_{\bd\alpha+\bd\beta,r,\bd s}\leq\|f(\bd z)\|_{\bd\alpha,r,\bd s}\|g(\bd z)\|_{\bd\beta,r,\bd s}$,\label{enu:nagumo_multi} 
\item $\|f(\bd z)\cdot g(\bd z)\|_{\bd\alpha,r,\bd s}\leq\|f(\bd z)\|_{\bd\alpha,r,\bd s}\|g(\bd z)\|_{\bd0,r,\bd s}$.\label{enu:nagumo_multi0} 
\end{enumerate}
\end{lemma}

Before we can move on to the proof of these facts, let us introduce
the following technical lemma:

\begin{lemma}\label{lem:binom} Let $p,q$ be positive natural numbers.
Then 
\begin{equation}
\sum_{k=0}^{n}\binom{k+p-1}{k}\binom{n-k+q-1}{n-k}=\binom{n+p+q-1}{n}\textrm{ for any }n\in\NN_{0}.\label{eq:binomial}
\end{equation}
\end{lemma}

\begin{proof} First let us note that for any $x\in\RR$ satisfying
$|x|<1$ it is true that 
\begin{equation}
\frac{1}{1-x}=\sum_{n=0}^{\infty}x^{n}.\label{eq:geom_series}
\end{equation}
After differentiating both sides of \eqref{eq:geom_series} $p-1$
times we receive 
\begin{equation*}
\frac{(p-1)!}{(1-x)^{p}}=\sum_{n=0}^{\infty}\frac{\left(n+p-1\right)!}{n!}x^{n}.
\end{equation*}
Hence, 
\begin{align*}
\frac{1}{\left(1-x\right)^{p+q}} & =\left(\sum_{n=0}^{\infty}\frac{\left(n+p-1\right)!}{n!(p-1)!}x^{n}\right)\left(\sum_{n=0}^{\infty}\frac{\left(n+q-1\right)!}{n!(q-1)!}x^{n}\right)\\
 & =\sum_{n=0}^{\infty}\sum_{k=0}^{n}\binom{k+p-1}{k}\binom{n-k+q-1}{n-k}x^{n}.
\end{align*}
On the other hand, we have 
\begin{equation*}
\frac{(p+q-1)!}{(1-x)^{p+q}}=\sum_{n=0}^{\infty}\frac{\left(n+p+q-1\right)!}{n!}x^{n},
\end{equation*}
which leads to the conclusion of this proof.\end{proof}

Let us now return to Lemma \ref{lem:nagumo_prop}.

\begin{proof}[Proof of Lemma \ref{lem:nagumo_prop}]

Let us fix multi-indices $\bd\alpha,\,\bd\beta\in\NN^{N}.$ By the
definitions of the Nagumo norms and $\Theta_{s}^{(a)}(x)$ we have
\begin{equation*}
f(\bd z)\ll\frac{\|f(\bd z)\|_{\bd\alpha,r,\bd s}}{r^{|\bd\alpha|}}\prod_{i=1}^{N}\sum_{n=0}^{\infty}\binom{n+\alpha_{i}-1}{n}^{s_{i}}\frac{z_{i}^{n}}{r^{n}}
\end{equation*}
and 
\begin{equation*}
g(\bd z)\ll\frac{\|g(\bd z)\|_{\bd\beta,r,\bd s}}{r^{|\bd\beta|}}\prod_{i=1}^{N}\sum_{n=0}^{\infty}\binom{n+\beta_{i}-1}{n}^{s_{i}}\frac{z_{i}^{n}}{r^{n}}.
\end{equation*}
From these two facts it follows that 
\begin{align*}
f(\bd z)\cdot g(\bd z) & \ll\frac{\|f(\bd z)\|_{\bd\alpha,r,\bd s}\|g(\bd z)\|_{\bd\beta,r,\bd s}}{r^{|\bd\alpha|+|\bd\beta|}}\prod_{i=1}^{N}\sum_{n=0}^{\infty}\sum_{k=0}^{n}\binom{k+\alpha_{i}-1}{k}^{s_{i}}\binom{n-k+\beta_{i}-1}{n-k}^{s_{i}}\frac{z_{i}^{n}}{r^{n}}.
\end{align*}
Using Lemma \ref{lem:binom} and the fact that $a^{s}+b^{s}\leq(a+b)^{s}$
for any $a,b>0$ and $s\geq1$, we conclude that 
\begin{equation*}
f(\bd z)\cdot g(\bd z)\ll\frac{\|f(\bd z)\|_{\bd\alpha,r,\bd s}\|g(\bd z)\|_{\bd\beta,r,\bd s}}{r^{|\bd\alpha|+|\bd\beta|}}\prod_{i=1}^{N}\sum_{n=0}^{\infty}\binom{n+\alpha_{i}+\beta_{i}-1}{k}^{s_{i}}\frac{z_{i}^{n}}{r^{n}},
\end{equation*}
which finishes the proof of (\ref{enu:nagumo_multi}).

To show (\ref{enu:nagumo_multi0}), first we note that $g(\bd z)=\sum_{\bd\gamma\in\NN_{0}^{N}}g_{\bd\gamma}\bd z^{\bd\gamma}$.
Analogously to the proof of (\ref{enu:nagumo_multi}) we can write:
\begin{align*}
f(\bd z)\cdot g(\bd z) & \ll\frac{\|f(\bd z)\|_{\bd\alpha,r,\bd s}}{r^{|\bd\alpha|}}\left(\prod_{i=1}^{N}\sum_{n=0}^{\infty}\binom{n+\alpha_{i}-1}{n}^{s_{i}}\frac{z_{i}^{n}}{r^{n}}\right)\left(\sum_{\bd\gamma\in\NN_{0}^{N}}\left|g_{\bd\gamma}\right|r^{|\bd\gamma|}\frac{\bd z^{\bd\gamma}}{r^{|\bd\gamma|}}\right)\\
 & =\frac{\|f(\bd z)\|_{\bd\alpha,r,\bd s}}{r^{|\bd\alpha|}}\left(\sum_{\bd\gamma\in\NN_{0}^{N}}\prod_{i=1}^{N}\binom{\gamma_{i}+\alpha_{i}-1}{\gamma_{i}}^{s_{i}}\frac{\bd z^{\bd\gamma}}{r^{|\bd\gamma|}}\right)\left(\sum_{\bd\gamma\in\NN_{0}^{N}}\left|g_{\bd\gamma}\right|r^{|\bd\gamma|}\frac{\bd z^{\bd\gamma}}{r^{|\bd\gamma|}}\right)\\
 & =\frac{\|f(\bd z)\|_{\bd\alpha,r,\bd s}}{r^{|\bd\alpha|}}\sum_{\bd\gamma\in\NN_{0}^{N}}\sum_{\bd\gamma'\leq\bd\gamma}\prod_{i=1}^{N}\binom{\gamma_{i}'+\alpha_{i}-1}{\gamma_{i}'}^{s_{i}}\left|g_{\bd\gamma-\bd\gamma'}\right|r^{|\bd\gamma-\bd\gamma'|}\frac{\bd z^{\bd\gamma}}{r^{|\bd\gamma|}}.
\end{align*}
Since for any $p\geq0$ and $k\leq n$ we have $\binom{k+p}{k}\leq\binom{n+p}{n}$,
it follows that 
\begin{align*}
f(\bd z)\cdot g(\bd z) & \ll\frac{\|f(\bd z)\|_{\bd\alpha,r,\bd s}}{r^{|\bd\alpha|}}\sum_{\bd\gamma\in\NN_{0}^{N}}\prod_{i=1}^{N}\binom{\gamma_{i}+\alpha_{i}-1}{\gamma_{i}}^{s_{i}}\left(\sum_{\bd\gamma'\leq\bd\gamma}\left|g_{\bd\gamma-\bd\gamma'}\right|r^{|\bd\gamma-\bd\gamma'|}\right)\frac{\bd z^{\bd\gamma}}{r^{|\bd\gamma|}}\\
 & \ll\frac{\|f(\bd z)\|_{\bd\alpha,r,\bd s}\|g(\bd z)\|_{\bd0,r,\bd s}}{r^{|\bd\alpha|}}\sum_{\bd\gamma\in\NN_{0}^{N}}\prod_{i=1}^{N}\binom{\gamma_{i}+\alpha_{i}-1}{\gamma_{i}}^{s_{i}}\frac{\bd z^{\bd\gamma}}{r^{|\bd\gamma|}}.
\end{align*}
\end{proof}

\begin{lemma} \label{lem:nagumo_der}Assume that $\bd\alpha\in\NN^{N}$
and $m_{1},\dots,m_{N}$ are regular Gevrey-type sequences of orders $s_{1},\dots,s_{N}$,
respectively, with every $s_{j}\geq 1$. Then 
\begin{equation*}
\|\partial_{m_{j},z_{j}}f(\bd z)\|_{\bd\alpha+\bd e_{j},r,\bd s}\leq C\alpha_{j}^{s_{j}}\|f(\bd z)\|_{\bd\alpha,r,\bd s}
\end{equation*}
for $j=1,\dots,N$, where $\bd e_{j}$ denotes a multi-index with
a $1$ in the $j$-th coordinate and zeros everywhere else.

\end{lemma}

\begin{proof}

Since $m_{j}$ is a regular Gevrey-type sequence, there exist $c,C>0$
such that 
\begin{equation*}
c(n+1)^{s_{j}}\leq\frac{m_{j}(n+1)}{m_{j}(n)}\leq C(n+1)^{s_{j}}\textrm{ for any }n\in\NN_{0}.
\end{equation*}
Let us now fix $\bd\alpha\in\NN^{N}$. It follows that: 
\begin{align*}
\partial_{m_{j},z_{j}}f(\bd z) & \ll\frac{\|f(\bd z)\|_{\bd\alpha,r,\bd s}}{r^{\alpha_{j}}(\alpha_{j}-1)!^{s_{j}}}\sum_{n=0}^{\infty}\frac{(n+\alpha_{j})!^{s_{j}}m_{j}(n+1)}{(n+1)!^{s_{j}}m_{j}(n)}\frac{z_{j}^{n}}{r^{n+1}}\prod_{i\neq j}\frac{1}{r^{\alpha_{i}}(\alpha_{i}-1)!^{s_{i}}}\Theta_{s_{i}}^{(\alpha_{i}-1)}\left(\frac{z_{i}}{r}\right)\\
 & \ll\frac{\|f(\bd z)\|_{\bd\alpha,r,\bd s}}{r^{\alpha_{j}+1}(\alpha_{j}-1)!^{s_{j}}}\sum_{n=0}^{\infty}\frac{C(n+\alpha_{j})!^{s_{j}}(n+1)^{s_{j}}}{(n+1)!^{s_{j}}}\frac{z_{j}^{n}}{r^{n}}\prod_{i\neq j}\frac{1}{r^{\alpha_{i}}(\alpha_{i}-1)!^{s_{i}}}\Theta_{s_{i}}^{(\alpha_{i}-1)}\left(\frac{z_{i}}{r}\right)\\
 & =\frac{C\|f(\bd z)\|_{\bd\alpha,r,\bd s}}{r^{\alpha_{j}+1}(\alpha_{j}-1)!^{s_{j}}}\sum_{n=0}^{\infty}\frac{(n+\alpha_{j})!^{s_{j}}}{n!^{s_{j}}}\frac{z_{j}^{n}}{r^{n}}\prod_{i\neq j}\frac{1}{r^{\alpha_{i}}(\alpha_{i}-1)!^{s_{i}}}\Theta_{s_{i}}^{(\alpha_{i}-1)}\left(\frac{z_{i}}{r}\right)\\
 & =C\alpha_{j}^{s_{j}}\|f(\bd z)\|_{\bd\alpha,r,\bd s}\prod_{i=1}^{N}\frac{1}{r^{\tilde{\alpha_{i}}}(\tilde{\alpha_{i}}-1)!^{s_{i}}}\Theta_{s_{i}}^{(\tilde{\alpha}_{i}-1)}\left(\frac{z_{i}}{r}\right),
\end{align*}
where $\bd{\tilde{\alpha}}=\bd\alpha+\bd e_{j}$. From this it follows
that $\left\Vert \partial_{m_{j},z_{j}}f(\bd z)\right\Vert _{\bd{\tilde{\alpha}},r,\bd s}\leq C\alpha_{j}^{s_{j}}\|f(\bd z)\|_{\bd\alpha,r,\bd s}$.

\end{proof}

\begin{lemma}\label{lem:counter_down}

Let $f\in\Oo\left(D_{R}^{N}\right).$ Then
\begin{equation*}
\|f(\bd z)\|_{\bd\alpha+\bd\beta,r,\bd s}\leq r^{|\bd\beta|}\|f(\bd z)\|_{\bd\alpha,r,\bd s}
\end{equation*}
for any $0<r<R$, $\bd\alpha\in\NN^{N}\cup\left\{ \bd0\right\} $
and $\bd\beta\in\NN^{N}$.

\end{lemma}

\begin{proof}

This fact follows directly from Lemma \ref{lem:nagumo_prop}. More
precisely, it is true that $\|f(\bd z)\|_{\bd\alpha+\bd\beta,r,\bd s}\leq\|1\|_{\bd\beta,r,\bd s}\|f(\bd z)\|_{\bd\alpha,r,\bd s}$.
Moreover, $\|1\|_{\bd\beta,r,\bd s}=r^{|\bd\beta|}$, because 
\begin{equation*}
1\ll\sum_{\bd\gamma\in\NN_{0}^{N}}\frac{z^{\bd\gamma}}{r^{|\bd\gamma|}}
\end{equation*}
and the same is not true for $(1-\epsilon)\sum_{\bd\gamma\in\NN_{0}^{N}}\frac{z^{\bd\gamma}}{r^{|\bd\gamma|}}$,
no matter how small $\epsilon>0$ we choose. Hence, $1=\frac{\|1\|_{\bd\beta,r,\bd s}}{r^{|\bd\beta|}}$,
which finishes the proof.\end{proof}

The following two lemmas connect Nagumo norms with the Gevrey order
of formal power series.

\begin{lemma}\label{lem:gevrey_norm}

Let us consider $\EE=\Oo\left(D_{R}^{N}\right)\cap C\left(\overline{D}_R^N\right)$ and $f(t,\bd z)=\sum_{n=0}^{\infty}f_{n}(\bd z)t^{n}\in\EE[[t]]_{w}$,
$w\geq0$. Then for any $0<r<R$ and $\bd\alpha\in\NN^{N}$ there
exist $A,B>0$ that satisfy 
\begin{equation*}
\|f_{n}(\bd z)\|_{n\bd\alpha,r,\bd s}\leq AB^{n}n!^{w}\textrm{ for any } n\in\NN_{0}.
\end{equation*}
\end{lemma}

\begin{proof}

Since $f(t,\bd z)$ is of Gevrey order $w$, for any $0<r<R$ there
exist $\tilde{A},\tilde{B}>0$ such that 
\begin{equation*}
\sup_{\bd\zeta\in D_{r}^{N}}\left|f_{n}(\bd\zeta)\right|\leq\tilde{A}\tilde{B}^{n}n!^{w}\textrm{ for }n\in\NN_{0}.
\end{equation*}
Moreover, since $f_{n}(\bd z)\in\Oo\left(D_{R}^{N}\right)$, we can
use Cauchy Inequalities for the power series of the form $f_{n}(\bd z)=\sum_{\bd\gamma\in\NN_{0}^{N}}f_{n,\bd\gamma}\bd z^{\bd\gamma}$.
More precisely, for any $0<r<R$ and $\bd\gamma\in\NN_{0}^{N}$ we
have
\begin{equation*}
\left|f_{n,\bd\gamma}\right|\leq\frac{\sup_{\bd\zeta\in D_{r}^{N}}\left|f_{n}(\bd\zeta)\right|}{r^{|\bd\gamma|}}.
\end{equation*}
Hence, 
\begin{equation*}
f_{n}(\bd z)\ll\sup_{\bd\zeta\in D_{r}^{N}}\left|f_{n}(\bd\zeta)\right|\sum_{\bd\gamma\in\NN_{0}^{N}}\frac{\bd z^{\bd\gamma}}{r^{|\bd\gamma|}}.
\end{equation*}

First suppose that $n\geq 1$. Seeing as $\binom{m}{l}\geq 1$ for any
$l\leq m$, we receive 
\begin{align*}
f_{n}(\bd z) & \ll\sup_{\bd\zeta\in D_{r}^{N}}\left|f_{n}(\bd\zeta)\right|\prod_{i=1}^{N}\sum_{k=0}^{\infty}\binom{k+n\alpha_{i}-1}{k}^{s_{i}}\frac{z_{i}^{k}}{r^{k}}\\
 & \ll\sup_{\bd\zeta\in D_{r}^{N}}\left|f_{n}(\bd\zeta)\right|\prod_{i=1}^{N}\frac{R^{n\alpha_{i}}}{r^{n\alpha_{i}}}\sum_{k=0}^{\infty}\binom{k+n\alpha_{i}-1}{k}^{s_{i}}\frac{z_{i}^{k}}{r^{k}}.\\
 & \ll\sup_{\bd\zeta\in D_{r}^{N}}\left|f_{n}(\bd\zeta)\right|R^{n|\bd\alpha|}\prod_{i=1}^{N}\frac{1}{r^{n\alpha_{i}}}\sum_{k=0}^{\infty}\binom{k+n\alpha_{i}-1}{k}^{s_{i}}\frac{z_{i}^{k}}{r^{k}}.
\end{align*}
From this it follows that $\left\Vert f_{n}(\bd z)\right\Vert _{n\bd\alpha,r,\bd z}\leq R^{n|\bd\alpha|}\sup_{\bd\zeta\in D_{r}^{N}}\left|f_{n}(\bd\zeta)\right|\leq\tilde{A}(\tilde{B}R^{|\bd\alpha|})^{n}n!^{w}$.

For $n=0$ let us notice that $f_{0}(\bd z)\in\Oo(D_{R}^{N})\cap C(\overline{D}_R^N)$ and
\begin{equation*}
f_{0}(\bd z)=\sum_{\bd\gamma\in\NN_{0}^{N}}f_{0,\bd\gamma}\bd z^{\bd\gamma}.
\end{equation*}
Then $|f_{0,\bd\gamma}|\leq\frac{\sup_{\bd\zeta\in D_{R}^{N}}|f_{0}(\bd\zeta)|}{R^{|\bd\gamma|}}$
for any $\bd\gamma\in\NN_{0}^{N}$ and 
\begin{equation*}
\left\Vert f_{0}(z)\right\Vert _{\bd0,r,\bd s}\leq\sup_{\bd\zeta\in D_{R}^{N}}|f_{0}(\bd\zeta)|\sum_{\bd\gamma\in\NN_{0}^{N}}\frac{r^{k}}{R^{k}}\leq\sup_{\bd\zeta\in D_{R}^{N}}|f_{0}(\bd\zeta)|\frac{R^{N}}{\left(R-r\right)^{N}}.
\end{equation*}

Hence it is enough to put $A=\max\left\{ \tilde{A},\,\sup_{\bd\zeta\in D_{R}^{N}}|f_{0}(\bd\zeta)|\frac{R^{N}}{\left(R-r\right)^{N}}\right\} $
and $B=\tilde{B}R^{|\bd\alpha|}$.\end{proof}

\begin{lemma}\label{lem:sup_norm}

Let $f\in\Oo\left(D_{R}^{N}\right).$ Then for any $0<\rho<r<R$ there
exists $A>0$ such that for any $\bd\alpha\in\NN^{N}\cup\left\{ \bd0\right\} $
the following inequality holds: 
\begin{equation*}
\sup_{\bd z\in D_{\rho}^{N}}\left|f(\bd z)\right|\leq A{}^{|\bd\alpha|}\|f(\bd z)\|_{\bd\alpha,r,\bd s}.
\end{equation*}

\end{lemma}

\begin{proof} When $\bd\alpha=\bd0$ it suffices to notice that for
certain $\theta_{1},\dots,\theta_{N}\in[0,\,2\pi)$ we have 
\begin{equation*}
\sup_{\bd z\in D_{\rho}^{N}}\left|f(\bd z)\right|=\left|\sum_{\bd\gamma\in\NN_{0}^{N}}f_{\bd\gamma}\prod_{i=1}^{N}\rho^{\gamma_{i}}e^{i\theta_{i}\gamma_{i}}\right|\leq\sum_{\bd\gamma\in\NN_{0}^{N}}\left|f_{\bd\gamma}\right|\rho^{|\bd\gamma|}\leq\sum_{\bd\gamma\in\NN_{0}^{N}}\left|f_{\bd\gamma}\right|r^{\bd\gamma}=\|f(\bd z)\|_{\bd0,r,\bd s}.
\end{equation*}

To prove the same for $\bd\alpha\in\NN^{N}$ first let us note that
for any $a,b\geq0$, $p\in\NN$ and $n\in\NN_{0}$ we have 
\begin{align}
\binom{n+p-1}{n}a^{n}b^{p-1} & \leq(a+b)^{n+p-1}.\label{eq:binom_ab}
\end{align}
In particular, if we set $\epsilon>0$ small enough that $\rho(1+\epsilon)^{|\bd s|-N}<r$
then \eqref{eq:binom_ab} holds for $a+b=1$ and $a^{-1}=1+\epsilon.$
Then for any $n\in\NN_{0}$ and $i=1,\dots,N$ we get 
\begin{equation*}
\binom{n+\alpha_{i}-1}{n}\leq\left(\frac{1+\epsilon}{\epsilon}\right)^{\alpha_{i}-1}\left(1+\epsilon\right)^{n}.
\end{equation*}
Now we can apply this fact to majorize $f(\bd z)$ in the following
way: 
\begin{align*}
f(\bd z) & \ll\frac{\|f(\bd z)\|_{\bd\alpha,r,\bd s}}{r^{|\bd\alpha|}}\prod_{i=1}^{N}\sum_{n=0}^{\infty}\binom{n+\alpha_{i}-1}{n}^{s_{i}}\frac{z_{i}^{n}}{r^{n}}\\
 & \ll\frac{\|f(\bd z)\|_{\bd\alpha,r,\bd s}}{r^{|\bd\alpha|}}\prod_{i=1}^{N}\sum_{n=0}^{\infty}\left(\frac{1+\epsilon}{\epsilon}\right)^{(\alpha_{i}-1)(s_{i}-1)}\left(1+\epsilon\right)^{n(s_{i}-1)}\binom{n+\alpha_{i}-1}{n}\frac{z_{i}^{n}}{r^{n}}\\
 & \ll\frac{\|f(\bd z)\|_{\bd\alpha,r,\bd s}}{r^{|\bd\alpha|}}\left(\frac{1+\epsilon}{\epsilon}\right)^{\bd s\cdot\bd\alpha}\prod_{i=1}^{N}\sum_{n=0}^{\infty}\binom{n+\alpha_{i}-1}{n}\left(\frac{z_{i}\left(1+\epsilon\right)^{s_{i}-1}}{r}\right)^{n}.
\end{align*}
In particular, the majorization above is true for any $\bd z\in D_{\rho}^{N}$
and so we can conclude that 
\begin{equation*}
\sup_{\bd z\in D_{\rho}^{N}}\left|f(\bd z)\right|\leq\frac{\|f(\bd z)\|_{\bd\alpha,r,\bd s}}{r^{|\bd\alpha|}}\left(\frac{1+\epsilon}{\epsilon}\right)^{\bd s\cdot\bd\alpha}\prod_{i=1}^{N}\sum_{n=0}^{\infty}\binom{n+\alpha_{i}-1}{n}\left(\frac{\rho\left(1+\epsilon\right)^{s_{i}-1}}{r}\right)^{n}.
\end{equation*}

Moreover, because of the properties of our chosen $\epsilon$, for
$i=1,\dots,N$ we have 
\begin{equation*}
\sum_{n=0}^{\infty}\frac{\left(n+\alpha_{i}-1\right)!}{n!}\left(\frac{\rho\left(1+\epsilon\right)^{s_{i}-1}}{r}\right)^{n}=\frac{d^{\alpha_{i}-1}}{du^{\alpha_{i}-1}}\left.\frac{1}{1-u}\right|_{u=\rho\left(1+\epsilon\right)^{s_{i}-1}r^{-1}}.
\end{equation*}
From this fact it follows that 
\begin{align*}
\sup_{\bd z\in D_{\rho}^{N}}\left|f(\bd z)\right| & \leq\frac{\|f(\bd z)\|_{\bd\alpha,r,\bd s}}{r^{|\bd\alpha|}}\left(\frac{1+\epsilon}{\epsilon}\right)^{\bd s\cdot\bd\alpha}\frac{r^{|\bd\alpha|}}{\left[r-(1+\epsilon)^{|\bd s|-N}\rho\right]^{|\bd\alpha|}}\\
 & \leq\|f(\bd z)\|_{\bd\alpha,r,\bd s}\left\{ \frac{(1+\epsilon)^{|\bd s|}}{\epsilon^{|\bd s|}\left[r-(1+\epsilon)^{|\bd s|-N}\rho\right]}\right\} ^{|\bd\alpha|}.
\end{align*}
Therefore it is enough to put $A=\max\left\{ 1,\,\frac{(1+\epsilon)^{|\bd s|}}{\epsilon^{|\bd s|}\left[r-(1+\epsilon)^{|\bd s|-N}\rho\right]}\right\} .$\end{proof}

\section{The linear Cauchy problem}\label{Sec4}

Suppose that $t\in\CC$ and $\bd z\in\CC^{N},$ $N\geq1$, and let
$\bd m=(m_{1},\dots,m_{N})$, where $m_{0},m_{1},\dots,m_{N}$ are
regular Gevrey-type sequences of non-negative orders $s_{0},s_{1},\dots,s_{N}$,
respectively. Then we define the Gevrey-type differential operator
$P\left(\partial_{m_{0},t},\partial_{\bd m,\bd z}\right)$ by the
formula: 
\begin{equation}
P\left(\partial_{m_{0},t},\partial_{\bd m,\bd z}\right)=\partial_{m_{0},t}^{M}+\sum_{(j,\bd\alpha)\in\Lambda}a_{j,\bd\alpha}(t,\bd z)\partial_{m_{0},t}^{j}\partial_{\bd m,\bd z}^{\bd\alpha}\label{eq:operator}
\end{equation}
with additional assumptions: 
\begin{enumerate}[(a)]
\item $\bd s=(s_{1},s_{2},\dots,s_{N})\in[1,\infty)^{N}$, \label{enu:s} 
\item $\Lambda\subset\NN_{0}^{1+N}$ is a finite set of indices, \label{enu:Lambda}
\item $\ord\left(a_{j,\bd\alpha}\right)\geq\max\left\{ 0,\,j-M+1\right\} $
for all $(j,\bd\alpha)\in\Lambda$.\label{enu:ord_t_a} 
\end{enumerate}
Our first step is to analyze the Newton polygon for (\ref{eq:operator}).
In this case it is given by
\begin{equation*}
N(P,s_{0},\bd s)=\textrm{conv}\left\{ \Delta(s_{0}M,\,-M)\cup\bigcup_{(j,\bd\alpha)\in\Lambda}\Delta\left(s_{0}j+\bd s\cdot\bd\alpha,\,\ord(a_{j,\bd\alpha})-j\right)\right\} .
\end{equation*}
Note that the first non-horizontal segment (if it exists) of the boundary
of $N(P,s_{0},\bd s)$ connects points with coordinates $(s_{0}M,-M)$
and $\left(s_{0}j^{*}+\bd s\cdot\bd\alpha^{*},\,\ord(a_{j^{*},\bd\alpha^{*}})-j^{*}\right)$
for a certain $(j^{*},\bd\alpha^{*})\in\Lambda$. The slope of this
segment is positive and given by the formula
\begin{equation*}
k_{1}=\frac{\ord(a_{j^{*},\bd\alpha^{*}})-j^{*}+M}{s_{0}(j^{*}-M)+\bd s\cdot\bd\alpha^{*}}.
\end{equation*}
Alternatively, we can define $k_{1}$ using the formula
\begin{equation*}
\frac{1}{k_{1}}=\max\left\{ 0,\,\max_{(j,\bd\alpha)\in\Lambda}\left\{ \frac{s_{0}(j-M)+\bd s\cdot\bd\alpha}{\ord\left(a_{j,\bd\alpha}\right)-j+M}\right\} \right\} .
\end{equation*}

Let us consider a linear Cauchy problem of the form 
\begin{equation}
\left\{ \begin{aligned}P\left(\partial_{m_{0},t},\partial_{\bd m,\bd z}\right)u(t,\bd z) & =f(t,\bd z)\\
\partial_{m_{0},t}^{j}u(0,\bd z) & =\phi_{j}(\bd z)\textrm{ for }0\leq j<M
\end{aligned}
\right.,\label{eq:cauchy}
\end{equation}
further assuming that:
\begin{enumerate}[(a)]
\setcounter{enumi}{3}
\item $f(t,\bd z)\in\Oo\left(D_{R}^{N}\right)[[t]]_{1/k_{1}}$ and $a_{j,\bd\alpha}(t,\bd z)\in\Oo\left(D_{R}^{N}\right)[[t]]_{1/k_{1}}$
for all $(j,\bd\alpha)\in\Lambda$, \label{enu:gevrey_order-Ndim} 
\item $\phi_{j}\in\Oo\left(D_{R}^{N}\right)$ for $0\leq j<M.$\label{enu:phi_j} 
\end{enumerate}
Our aim is to prove the following theorem:

\begin{thm}\label{thm:gevrey}

Under assumptions (\ref{enu:s})-(\ref{enu:phi_j}) listed above
the formal solution $\hat{u}(t,\bd z)=\sum_{n=0}^{\infty}u_{n}(\bd z)t^{n}$
of \eqref{eq:cauchy} is of the Gevrey order $\frac{1}{k_{1}}.$

\end{thm}

Before moving on to the proof of this theorem, we shall show that
the following proposition holds true:

\begin{prop}\label{prop:norm_est}

Suppose that $\bd\alpha_{0}=(\alpha_{0,1},\alpha_{0,2},\dots,\alpha_{0,N})$
is a multi-index such that 
\begin{equation*}
\alpha_{0,k}:=\left\lfloor \max_{(j,\bd\alpha)\in\Lambda}\frac{\alpha_{k}}{\emph{ord}_t\left(a_{j,\bd\alpha}\right)-j+M}\right\rfloor +1\textrm{ for }k=1,2,\dots,N
\end{equation*}
and that $\hat{u}(t,\bd z)=\sum_{n=0}^{\infty}u_{n}(\bd z)t^{n}$
is a formal solution of \eqref{eq:cauchy}. Then for any $0<r<R$
there exist constants $A,B>0$ such that 
\begin{equation}
\|u_{n}(z)\|_{n\bd\alpha_{0},r,\bd s}\leq AB^{n}n!^{1/k_{1}}\textrm{ for any }n\in\NN_{0}.\label{eq:norm_est}
\end{equation}
\end{prop}

\begin{proof}

Without any loss of generality we may assume that $u_{0}(\bd z)\equiv0$.
Otherwise it is enough to use substitution $v(t,\bd z):=u(t,\bd z)-\phi_{0}(\bd z)$.

Now we can proceed to show \eqref{eq:norm_est} by induction on $n$.
It is easy to see that \eqref{eq:norm_est} holds for $n\leq M-1$.
This fact follows directly from the properties of functions $\phi_{1}(\bd z),\dots,\phi_{M-1}(\bd z)$
and Lemma \ref{lem:gevrey_norm}. Let us then show that if $n\geq M$
and the proposition (with $n$ replaced by $i$) holds for any $i<n$
then it is also true for $n$.

Let $t^{M}f(t,\bd z)=\sum_{n=M}^{\infty}f_{n}(\bd z)t^{n}$ and $t^{M-j}a_{j,\bd\alpha}(t,\bd z)=\sum_{n=q_{j,\bd\alpha}}^{\infty}a_{j,\bd\alpha,n}(\bd z)t^{n}$
with $q_{j,\bd\alpha}=\ord\left(a_{j,\bd\alpha}(t,\bd z)\right)-j+M$
for every $(j,\bd\alpha)\in\Lambda.$ Using these facts we find a
recurrence relation describing $u_{n}(\bd z)$ for $n\geq M$: 
\begin{equation}
u_{n}(\bd z)=\frac{m_{0}(n-M)}{m_{0}(n)}\left[f_{n}(\bd z)-\sum_{(j,\bd\alpha)\in\Lambda}\sum_{p=q_{j,\bd\alpha}}^{n}a_{j,\bd\alpha,p}(\bd z)\frac{m_{0}(n-p)}{m_{0}(n-p-j)}\partial_{\bd m,\bd z}^{\bd\alpha}u_{n-p}(\bd z)\right].\label{eq:recurrence}
\end{equation}
From this point on we shall use a convention that a term $\frac{m_{0}(n-p)}{m_{0}(n-p-j)}$
disappears whenever $n-p-j<0$.

Let us fix $0<r<R$. By assumption \eqref{enu:gevrey_order-Ndim}
and Lemma \ref{lem:gevrey_norm} there exist positive constants $F,\,G$,
for which $\|f_{n}(\bd z)\|_{n\bd\alpha_{0},r,\bd s}\leq FG^{n}n!^{1/k_{1}}$.
Moreover, from regularity of $m_{0}$ and $m_{1},\dots,m_{N}$, there
exist constants $c,C>0$ such that 
\begin{equation*}
c(n+1)^{s_{j}}\leq\frac{m_{j}(n+1)}{m_{j}(n)}\leq C(n+1)^{s_{j}}\textrm{ for any }n\in\NN_{0}\textrm{ and }j=0,1,\dots,N.
\end{equation*}
Hence, we have 
\begin{multline*}
\|u_{n}(\bd z)\|_{n\bd\alpha_{0},r,\bd s}\leq\frac{c^{-M}(n-M)!^{s_{0}}}{n!^{s_{0}}}\left[FG^{n}n!^{1/k_{1}}+\sum_{(j,\bd\alpha)\in\Lambda}\sum_{p=q_{j,\bd\alpha}}^{n}C^{j}\frac{(n-p)!^{s_{0}}}{(n-p-j)!^{s_{0}}}\right.\\
\left.\times\|a_{j,\bd\alpha,p}(\bd z)\partial_{\bd m,\bd z}^{\bd\alpha}u_{n-p}(\bd z)\|_{n\bd\alpha_{0},r,\bd s}\right].
\end{multline*}
It is easy to see that the first term on the right-hand side of the
inequality above can be bounded by $\frac{1}{2}AB^{n}n!^{1/k_{1}}$
if we put $A\geq2c^{-M}F$ and $B\geq G$.

Now we can analyze the term: 
\begin{equation*}
I=\frac{c^{-M}(n-M)!^{s_{0}}}{n!^{s_{0}}}\sum_{(j,\bd\alpha)\in\Lambda}\sum_{p=q_{j,\bd\alpha}}^{n}C^{j}\frac{(n-p)!^{s_{0}}}{(n-p-j)!^{s_{0}}}\|a_{j,\bd\alpha,p}(\bd z)\partial_{\bd m,\bd z}^{\bd\alpha}u_{n-p}(\bd z)\|_{n\bd\alpha_{0},r,\bd s},
\end{equation*}
By Lemma \ref{lem:nagumo_prop} we have 
\begin{equation*}
\|a_{j,\bd\alpha,p}(\bd z)\partial_{\bd m,\bd z}^{\bd\alpha}u_{n-p}(\bd z)\|_{n\bd\alpha_{0},r,\bd s}\leq\|a_{j,\bd\alpha,p}(\bd z)\|_{p\bd\alpha_{0}-\bd\alpha,r,\bd s}\|\partial_{\bd m,\bd z}^{\bd\alpha}u_{n-p}(\bd z)\|_{(n-p)\bd\alpha_{0}+\bd\alpha,r,\bd s}.
\end{equation*}
By Lemma \ref{lem:counter_down} and
the definition of $\bd\alpha_{0}$ we also get 
\begin{equation*}
\|a_{j,\bd\alpha,p}(\bd z)\|_{n\bd\alpha_{0}-\bd\alpha,r,\bd s}\leq R^{q_{j,\bd\alpha}|\bd\alpha_{0}|-|\bd\alpha|}\|a_{j,\bd\alpha,p}(\bd z)\|_{(p-q_{j,\bd\alpha})\bd\alpha_{0},r,\bd s}.
\end{equation*}
Note that $q_{j,\bd\alpha}\geq1$ for every $(j,\bd\alpha)\in\Lambda$
and so we can write that $t^{M-j}a_{j,\bd\alpha}(t,\bd z)=t^{q_{j,\bd\alpha}}\sum_{n=q_{j,\bd\alpha}}^{\infty}a_{j,\bd\alpha,n}(\bd z)t^{n-q_{j,\bd\alpha}}$
and treat every $a_{j,\bd\alpha,n}(\bd z)$ like the $\left(n-q_{j,\bd\alpha}\right)$-th
coefficient of the formal power series. Hence, using Lemma \ref{lem:gevrey_norm} we receive
\begin{equation*}
\|a_{j,\bd\alpha,p}(\bd z)\|_{(p-q_{j,\bd\alpha})\bd\alpha_{0},r,\bd s}\leq \tilde{K}L^{p}\left(p-q_{j,\bd\alpha}\right)!^{1/k_{1}}
\end{equation*}
for a certain pair of positive constants $\tilde{K},\,L$. If we put $K=R^{q_{j,\bd\alpha}|\bd\alpha_{0}|}\tilde{K}$ then
\begin{equation*}
\|a_{j,\bd\alpha,p}(\bd z)\|_{n\bd\alpha_{0}-\bd\alpha,r,\bd s}\leq KL^{p}\left(p-q_{j,\bd\alpha}\right)!^{1/k_{1}}.
\end{equation*}

Moreover, by Lemma \ref{lem:nagumo_der} we have 
\begin{multline*}
\|\partial_{\bd m,\bd z}^{\bd\alpha}u_{n-p}(\bd z)\|_{(n-p)\bd\alpha_{0}+\bd\alpha,r,\bd s}\leq C^{|\bd\alpha|}\left|\bd\alpha_{0}\right|^{\bd s\cdot\bd\alpha}\|u_{n-p}(\bd z)\|_{(n-p)\bd\alpha_{0},r,\bd s}\\
\times\prod_{k=1}^{N}\left(n-p+\frac{\alpha_{k}-1}{\alpha_{0,k}}\right)^{s_{k}}\left(n-p+\frac{\alpha_{k}-2}{\alpha_{0,k}}\right)^{s_{k}}\dots\left(n-p\right)^{s_{k}}
\end{multline*}
and, by the inductive assumption,
\begin{equation*}
\|u_{n-p}(\bd z)\|_{(n-p)\bd\alpha_{0},r,\bd s}\leq AB^{n-p}(n-p)!^{1/k_{1}}.
\end{equation*}

Furthermore, let us observe that 
\begin{equation*}
\left(p-q_{j,\bd\alpha}\right)!^{1/k_{1}}(n-p)!^{1/k_{1}}\leq\frac{n!^{1/k_{1}}}{(n-p+1)^{1/k_{1}}\dots(n-p+q_{j,\bd\alpha})^{1/k_{1}}}\leq\frac{n!^{1/k_{1}}}{(n-p+1)^{q_{j,\bd\alpha}/k_{1}}}
\end{equation*}
and 
\begin{equation*}
n-p+\frac{l}{\alpha_{0,k}}\leq n-p+\frac{q_{j,\bd\alpha}l}{\alpha_{k}}\textrm{ for any }0\leq l\leq\alpha_{k}-1,\,1\leq k\leq N.
\end{equation*}
Hence, 
\begin{multline*}
I\leq{}AB^{n}n!^{1/k_{1}}\sum_{(j,\alpha)\in\lambda}c^{-M}KC^{j+|\alpha|}\left|\bd\alpha_{0}\right|^{\bd s\cdot\bd\alpha}\frac{(n-M)!^{s_{0}}(n-p)!^{s_{0}}\prod_{k=1}^{N}\prod_{l=0}^{\alpha_{k}-1}\left(n-p+\left\lceil \frac{q_{j,\bd\alpha}l}{\alpha_{k}}\right\rceil \right)^{s_{k}}}{n!^{s_{0}}(n-p-j)!^{s_{0}}(n-p+1)^{q_{j,\bd\alpha}/k_{1}}}\\
\times\sum_{p=q_{j,\bd\alpha}}^{n}\left(\frac{L}{B}\right)^{p}
\end{multline*}
\begin{align*}
&\leq{}AB^{n}n!^{1/k_{1}}\sum_{(j,\alpha)\in\lambda}c^{-M}KC^{j+|\alpha|}\left|\bd\alpha_{0}\right|^{\bd s\cdot\bd\alpha}\frac{(n-p)^{s_{0}j}\left(n-p+q_{j,\bd\alpha}\right)^{\bd s\cdot\bd\alpha}}{(n-M+1)^{s_{0}M}(n-p+1)^{s_{0}(j-M)+\bd s\cdot\bd\alpha}}\sum_{p=q_{j,\bd\alpha}}^{n}\left(\frac{L}{B}\right)^{p}\\
&\leq{} AB^{n}n!^{1/k_{1}}\sum_{(j,\alpha)\in\lambda}c^{-M}KC^{j+|\alpha|}\left|\bd\alpha_{0}\right|^{\bd s\cdot\bd\alpha}\left(\frac{n-p+q_{j,\bd\alpha}}{n-p+1}\right)^{\bd s\cdot\bd\alpha}\left(\frac{n-p+1}{n-M+1}\right)^{s_{0}M}\sum_{p=q_{j,\bd\alpha}}^{n}\left(\frac{L}{B}\right)^{p}.
\end{align*}

Moreover, 
\begin{equation*}
\frac{n-p+q_{j,\bd\alpha}}{n-p+1}\leq1+\frac{q_{j,\bd\alpha}-1}{n-p+1}\leq q_{j,\bd\alpha}\textrm{ for any }1\leq p\leq n
\end{equation*}
and 
\begin{equation*}
\frac{n-p+1}{n-M+1}\leq\frac{1}{1-\frac{M-1}{n}}\leq M\textrm{ for any }n\geq M,
\end{equation*}

By combining facts listed above we receive 
\begin{align*}
I & \leq AB^{n}n!^{1/k_{1}}\sum_{(j,\alpha)\in\Lambda}c^{-M}KC^{j+|\alpha|}\left|\bd\alpha_{0}\right|^{\bd s\cdot\bd\alpha}M^{Ms_{0}}q_{j,\bd\alpha}^{\bd s\cdot\bd\alpha}\sum_{p=q_{j,\bd\alpha}}^{n}\left(\frac{L}{B}\right)^{p}\\
 & \leq AB^{n}n!^{1/k_{1}}\sum_{(j,\alpha)\in\Lambda}c^{-M}KC^{j+|\alpha|}\left|\bd\alpha_{0}\right|^{\bd s\cdot\bd\alpha}M^{Ms_{0}}q_{j,\bd\alpha}^{\bd s\cdot\bd\alpha}\sum_{p=q_{j,\bd\alpha}}^{\infty}\left(\frac{L}{B}\right)^{p}\\
 & =AB^{n}n!^{1/k_{1}}\sum_{(j,\alpha)\in\Lambda}c^{-M}KC^{j+|\alpha|}\left|\bd\alpha_{0}\right|^{\bd s\cdot\bd\alpha}M^{Ms_{0}}q_{j,\bd\alpha}^{\bd s\cdot\bd\alpha}\frac{L}{B-L}.
\end{align*}
It remains to notice that the last sum can be bounded from above by
$\frac{1}{2}$ for sufficiently large $B>L$.\end{proof}

Now we can go back to the proof of Theorem \ref{thm:gevrey}.

\begin{proof}[Proof of Theorem \ref{thm:gevrey}]

Using Proposition \ref{prop:norm_est} and Lemma \ref{lem:sup_norm}
we conclude that for any $n\in\NN_{0}$ and $\rho<r$ there exist
constants $\tilde{A}$, $A$, and $B$ such that 
\begin{equation*}
\sup_{\bd z\in D_{\rho}^{N}}|u_{n}(\bd z)|\leq\tilde{A}^{n|\bd\alpha_{0}|}\|u_{n}(z)\|_{n\bd\alpha_{0},r,\bd s}\leq A\left(\tilde{A}^{|\bd\alpha_{0}|}B\right)^{n}n!^{1/k_{1}}.
\end{equation*}
Hence, $\hat{u}(t,\bd z)\in\Oo\left(D_{R}^{N}\right)[[t]]_{1/k_{1}}$.\end{proof}

\section{Suggestions for further research\label{Sec5}}
It might be possible to extend the result from Theorem \ref{thm:gevrey} to a certain class of non-linear Gevrey-type differential equations. We would like to achieve this by using methods inspired by the approach presented in this paper.

Moreover, it is worth mentioning that an idea similar to the Gevrey-type sequences has already been considered in \cite{LMiSu}. Sequences and operators defined there are similar to the ones  presented in Definitions \ref{def:gevrey_seq} and \ref{def:gevrey_op}, with the Gevrey sequence $n!^{s}$ replaced with any given sequence $M_{n}$ of positive real numbers. Under certain additional assumptions concerning the sequence
$M_{n}$ the authors were able to analyze the Gevrey properties of a certain class of linear equations with time-dependent coefficients using the approach similar to \cite{TY} and \cite{MicS}. A question remains open, whether those findings can be extended to the case when the coefficients depend on both $t\in\CC$ and $\bd z\in\CC^N$.

\end{document}